\documentclass[12pt,a4paper]{article}
\usepackage[utf8]{inputenc}
\usepackage{amsmath}
\usepackage{amsthm}
\usepackage{amsfonts}
\usepackage{amssymb}
\usepackage{dsfont}
\usepackage{graphicx}
\usepackage[left=2cm,right=2cm,top=2cm,bottom=2cm]{geometry}

\newtheorem{proposition}{Proposition}[section]
\newtheorem{lemma}{Lemma}[section]
\newtheorem{theorem}{Theorem}[section]

\theoremstyle{definition}

\newtheorem*{solution*}{Solution}

 \newcommand{\bbar}[1]{\setbox0=\hbox{$#1$}\dimen0=.2\ht0 \kern\dimen0 \overline{\kern-\dimen0 #1}}

 \DeclareMathOperator{\End}{\ensuremath{\mathcal{E}\kern-.125em\mathpzc{nd}}}

 \DeclareMathOperator{\Hom}{\mathcal{H}\kern-.125em\mathpzc{om}}
 
 \DeclareMathOperator{\id}{id}
 \DeclareMathOperator{\im}{im}

 \renewcommand{\setminus}{\smallsetminus}

 \newcommand{\udot}{\ensuremath{{\lower .183333em \hbox{\LARGE \kern -.05em$\cdot$}}}}

 \DeclareMathOperator{\Hess}	{Hess}
  \DeclareMathOperator{\sgn}	{sgn}
  \DeclareMathOperator{\ind}	{ind}
  \DeclareMathOperator{\Ind}	{Ind}

 \newcommand{\cL}{\mathcal{L}}
 \newcommand{\cM}{\mathcal{M}}

 \newcommand{\cU}{\mathcal{U}}

 \newcommand{\R}{\mathbb{R}}

 \newcommand{\fL}{\mathfrak{L}}

 \newcommand{\p}{\partial}

\title{Symplectic Geometry of Constrained Optimization}
\author{A. Agrachev and I. Beschastnyi}

\begin{document}
\maketitle

There are the notes of rather informal lectures given by the first co-author in UPMC, Paris, in January 2017.
Practical goal is to explain how to compute or estimate the Morse index of the second variation. Symplectic geometry allows to effectively do it even for very degenerate problems with complicated constraints. Main geometric and analytic tool is the appropriately rearranged Maslov index.

In these lectures, we try to emphasize geometric structure and omit analytic routine. Proofs are often substituted by informal explanations but a well-trained mathematician will easily re-write them in a conventional way.

\section{Lecture 1}

\subsection{First variations in finite dimensions}
\label{sec:first_var}
Our goal in these notes will be to develop a general machinery for optimization problems using the language and results from symplectic geometry.

We start by discussing the Lagrange multiplier rule in the finite dimensional setting. Let $U$ be a finite dimensional manifold, $\varphi: U \to \R$ a smooth function and $\Phi: U \to M$ a smooth submersion onto a finite dimensional manifold $M$. We would like to find critical points of $\varphi$ restricted to the level sets of $\Phi$. It is well known that this can be done via the Lagrange multiplier rule.

\begin{theorem}[Lagrange multiplier rule]
A point $u \in U$ is a critical point of $\varphi|_{\Phi^{-1}(q)}$, $q\in M$ if and only if there exists a covector $\lambda \in T^*_q M$, s.t.
\begin{equation}
\label{eq:Lag}
\begin{array}{c}
d_u \varphi = \lambda D_u \Phi, \\
\Phi(u) = q.
\end{array}
\end{equation}
\end{theorem}

This fact has an important geometric implication. Suppose that we have locally a smooth function $q \to u(q)$, s.t. each $u$ satisfies the Lagrange multiplier rule. Then the covectors $\lambda$ corresponding to $u(q)$ are just the values of the differential of the cost function $c(q) = \varphi(u(q))$. Indeed, we can differentiate the constraint equation $\Phi(u(q)) = q$ to get
$$
\id = \left( D_u \Phi \right) \frac{\p u}{\p q} \qquad \Rightarrow \qquad  \lambda = \left( \lambda D_u \Phi \right) \frac{\p u}{\p q} = \left( d_u \varphi \right) \frac{\p u}{\p q} = d_q c.
$$

So if we choose a branch of $u(q)$, the set of correspondent Lagrange multipliers is a graph of a differential of a smooth function that is a Lagrangian submanifold of the symplectic manifold $T^*M$. Let us briefly recall the symplectic terminology.

\subsection{Basic symplectic geometry}
\label{sec:bas_symp}

In this subsection we give a some basic definitions from symplectic geometry. For further results and proofs see~\cite{gosson,sal_mcd}.

A \emph{symplectic space} is a pair $(W,\sigma)$ of an even-dimensional vector space $W$ and a skew-symmetric non-degenerate bilinear form $\sigma$. One can always choose a basis in $W$, s.t. $\sigma$ is of the form
$$
\sigma(\lambda_1,\lambda_2) = \lambda_1^T J \lambda_2, \qquad \lambda_i \in W
$$
where
$$
J = \begin{pmatrix}
0 & \id \\
-\id & 0
\end{pmatrix}.
$$
Such a basis is called a \emph{Darboux basis}.

A \emph{symplectic map} is a linear map $F: W \to W$ that preserves the symplectic structure, i.e.
$$
\sigma(F\lambda_1,F\lambda_2) = \sigma(\lambda_1,\lambda_2).
$$
In a Darboux basis we can equivalently write
$$
F^T J F = J.
$$

We define the \emph{skew-orthogonal complement} of a subspace $\Gamma$ in a symplectic space $W$ as a subspace
$$
\Gamma^\angle = \{\lambda \in W \,: \sigma(\lambda,\mu)=0, \forall \mu \in \Gamma\,\}.
$$
One has the following special situations
\begin{itemize}
\item If $\Gamma \subset \Gamma^\angle$, then $\Gamma$ is called \emph{isotropic};
\item If $\Gamma \supset \Gamma^\angle$, then $\Gamma$ is called \emph{coisotropic};
\item If $\Gamma = \Gamma^\angle$, then $\Gamma$ is called \emph{Lagrangian}.
\end{itemize}

From the definition we can see, that $\Gamma$ is isotropic if and only if the restriction $\sigma|_\Gamma$ vanishes. Since $\sigma$ is non-degenerate, we have
$$
\dim \Gamma + \dim \Gamma^\angle = \dim W.
$$
Therefore a subspace $\Gamma $ is Lagrangian if and only if $\Gamma$ is isotropic and has dimension $(\dim W)/2$. Any one-dimensional subspace is isotropic by the skew-symmetry of $\sigma$. For the same reasons any codimension one subspace is coisotropic.

In a Darboux basis each vector $\lambda\in W$ has coordinates $(p_1,...,p_n,q^1,...,q^n) = (p,q)$, where $n = (\dim W)/2$. Then the subspaces defined by equations $p=0$ or $q=0$ are Lagrangian. To construct more examples we can consider a graph $(p,Sp)$ of a linear map $S$ between those subspaces. Then it is easy to check that $(p,Sp)$ gives a Lagrangian subspace if and only if $S$ is symmetric.

There exists a close relation between symplectic maps and Lagrangian subspaces. Given $(W,\sigma)$ we can construct a new symplectic space $(W\times W, (-\sigma) \oplus \sigma)$ of double dimension. It can be used to give an alternative definition of a symplectic map.
\begin{proposition}
\label{prop:symp_lagr}
Let $F: W\to W$ be a linear map. $F$ is symplectic if and only if the graph of $F$ in $(W\times W, (-\sigma)\oplus \sigma)$ is Lagrangian.
\end{proposition}

We can extended all these definitions to the non-linear setting. A \emph{symplectic manifold} is a pair $(W,\sigma)$, where $W$ is a smooth manifold and $\sigma$ is a closed non-degenerate differential two-form. Similar to the linear case, one can show that locally all symplectic manifolds have the same structure.
\begin{theorem}[Darboux]
For any point $x$ of a symplectic manifold $(W,\sigma)$ one can find a neighbourhood $U$ and a local diffeomorphism $\psi : U \to \R^{2n}$, s.t.
$$
\sigma =\psi^*(dp_i\wedge dq^i),
$$
where $(p,q)$ are coordinates in $\R^{2n}$.
\end{theorem}

Note that a tangent space $T_x W$ has naturally a structure of a symplectic space. Therefore we can say that a submanifold $N\subset M$ is \emph{isotropic/coisotropic/Lagrangian} if the same property is true for each subspace $T_x N \subset T_x W$ for all $x\in N$. Similarly to the linear case a submanifold $N$ is isotropic if and only if $\sigma|_N = 0$ and Lagrangian if additionally $\dim N = (\dim W)/2$.

A \emph{symplectomorphism} of $(W,\sigma)$ is a smooth map $f: W\to W$, that preservers the symplectic structure, i.e.
$$
f^*\sigma = \sigma.
$$

Given a smooth function $h:W\to\mathbb R$, a {\it Hamiltonian vector field} $\vec h$ is defined by the identity
$dh=\sigma(\cdot,\vec h)$. The flow generated by the {\it Hamiltonian system} $\dot x=\vec h(x)$ preserves the symplectic structure. In Darboux coordinates, Hamiltonian system has the form:
$$
\dot p=-\frac{\p h}{\p q},\quad \dot q=\frac{\p h}{\p p}.
$$

The non-linear analogue of Proposition~\ref{prop:symp_lagr} holds as well
\begin{proposition}
A diffeomorphism $f: W\to W$ of a symplectic manifold $(W,\sigma)$ is a symplectomorphism if and only if the graph of $f$ in $(W\times W, (-\sigma) \oplus \sigma)$ is a Lagrangian submanifold.
\end{proposition}

The most basic and important examples of symplectic manifolds are the cotangent bundles $T^*M$. To define invariantly the symplectic form we use the projection map $\pi: T^*M \to M$. It's differential is a well defined map $\pi_*: T (T^*M) \to TM$. We can define the \emph{Liouville one-form} $s\in \Lambda^1(T^* M)$  at $\lambda \in T^*M$ as
$$
s_\lambda = \lambda \circ \pi_*.
$$
Then the canonical symplectic form on $T^*M$ is simply given by the differential $\sigma = ds$.

In local coordinates $T^*M$ is locally diffeomorphic to $\R^n \times \R^n$ with coordinates $(p,q)$, where $q$ are coordinates on the base and $p$ are coordinates on the fibre. In these coordinates the Liouville form $s$ is written as $s = p_idq^i$. Thus $(p,q)$ are actually Darboux coordinates. We can use this fact to construct many Lagrangian manifolds. Namely

\begin{proposition}
 Let $S: M \to \R$ be a smooth function. Then the graph of the differential $d_q S$ is a Lagrangian submanifold in $T^* M$.
 \end{proposition}
The proof is a straightforward computation in the Darboux coordinates and follows from the commutativity of the second derivative of $S$.

We have seen in the previous subsection that the set of Lagrange multipliers is often a graph of the differential of a smooth function. One can reformulate this by saying that that the set of Lagrange multipliers is actually a Lagrangian submanifold. In the next sections we will see that this is a rather general fact, but the resulting ``Lagrangian set" can be quite complicated. So we will linearise our problem and extract optimality information from the behaviour of what is going to be ``tangent spaces'' to this ``Lagrangian sets''. This way we obtain a geometric theory of second variation that is applicable to a very large class of optimization problems.

\subsection{First variation for classical calculus of variations}
\label{sec:calc_var}
Let us consider a geometric formulation of the classical problem of calculus of variations, which is an infinite dimensional optimisation problem. We denote by $\cU$ the set of Lipschitzian  curves $\gamma: [\tau,t] \to M$, where $M$ is a finite dimensional manifold. Assume that this set is endowed with a nice topology of a Hilbert manifold. We consider a family of functionals
$$
J_\tau^t:\gamma\mapsto   \int_\tau^t l(\gamma(s),\dot{\gamma}(s))ds.
$$

We define the evaluation map $F_s: \cU \to M$, that takes a curve and returns a point on it at a time $s\in[\tau,t]$, i.e. $F_s(\gamma) = \gamma_s$. We look for the critical points of the restriction
$$
J_\tau^t|_{\{F_\tau^{-1}(q_\tau),F_t^{-1}(q_t)\}}.
$$
So we apply the Lagrange multiplier rule to $\varphi = J_\tau^t$ with $\Phi = (F_\tau,F_t)$ and find that there exists a pair $(-\lambda_\tau,\lambda_t) \in T_{q(\tau)}^* M \times T_{q(t)}^* M$, s.t.
\begin{equation}
\label{eq:lag_inf}
d_\gamma J_\tau^t = \lambda_t D_\gamma F_t - \lambda_\tau D_\gamma F_\tau.
\end{equation}

The evaluation map is obviously a submersion. This fact implies that ones we fix $\gamma$ and one of the covectors $(-\lambda_\tau,\lambda_t)$, the other one will be determined automatically. Indeed, suppose for example that $(-\lambda_\tau,\lambda_t)$ and $(-\lambda_\tau,\lambda'_t)$ both satisfy the Lagrange multiplier rule. Then in addition to (\ref{eq:lag_inf}) we have
$$
d_\gamma J_\tau^t = \lambda'_t D_\gamma F_t - \lambda_\tau D_\gamma F_\tau.
$$
We subtract this equation from (\ref{eq:lag_inf}) and obtain
$$
(\lambda'_t - \lambda_t) D_\gamma F_t = 0,
$$
which gives a contradiction with the fact that $F_t$ is a submersion.

As in the finite dimensional case under some regularity assumptions the Lagrange multipliers form a Lagrangian submanifold in $T^*M \times T^*M$ endowed with a symplectic form $(-\sigma) \oplus \sigma$. Since each $\lambda_\tau$ determines a unique $\lambda_t$, we get that this Lagrangian submanifold can be identified with a graph of some map $A^t_\tau: T^*M \to T^*M$. Such a graph is a Lagrangian submanifold if and only if $A_\tau^t$ is symplectic. Moreover, the identity $J_\tau^t=J_\tau^s+J_s^t$ implies that $A_\tau^t$ is actually a symplectic flow, i.e. $A^t_\tau = A^t_s \circ A^s_\tau$, where $s\in[\tau,t],\ A_\tau^\tau=\mathrm{Id}$.

Since we have a symplectic flow, it should come from a Hamiltonian system. Let us find an expression for the corresponding Hamiltonian. We introduce some local coordinates $(p,q)$ on $T^*M$. Then our extremal curve $\gamma$ is given by a map $s\to q(s)$. We denote $\dot{q} = v$ and write down the equation (\ref{eq:lag_inf})
$$
\int_\tau^t\left( \frac{\p l}{\p q} dq_s + \frac{\p l}{\p v} dv_s \right)ds = p_t dq_t - p_\tau dq_\tau.
$$
We differentiate this expression w.r.t. time $t$:
$$
\frac{\p l}{\p q} dq_t + \frac{\p l}{\p v} dv_t = \dot p_t dq_t + p_t dv_t
$$
Then we obtain
\begin{align}
\dot{q}_t &= v_t, \nonumber\\
\dot{p}_t &= \frac{\p l}{\p q},\label{eq:ham}\\
p_t &= \frac{\p l}{\p v} . \nonumber
\end{align}

The first two equations can be seen to be a Hamiltonian system
\begin{align}
\dot{q}_t &= \frac{\p H}{\partial p}, \nonumber\\
\dot{p}_t &= -\frac{\p H}{\p q},
\end{align}
 with a Hamiltonian
$$
H(v,p,q) = \langle p,v \rangle - l(q,v)
$$
and the third equation gives a condition
$$
\frac{\p H}{\p v} = 0.
$$
If the second derivative of the Hamiltonian $H$ w.r.t. $v$ is non-degenerate, then by the inverse function theorem we can locally resolve this condition to obtain a function $v = v(p,q)$. Substituting it in $H(p,q,v)$, we get an autonomous Hamiltonian system with a Hamiltonian $H(p,q,v(p,q))$.

\subsection{Second variation}
\label{sec:sec_var}

Now we are going back to the general setting with $\varphi: U \to \R$ and constraints $\Phi: U \to M$. Once we have found a critical point $u\in U$, we would like to study the index of the Hessian, which is a quadratic form
$$
\Hess_u\varphi|_{\Phi^{-1}(q)} = \ker D_u\Phi \times \ker D_u\Phi \to \R.
$$
We can write an explicit expression for the Hessian without resolving the constraints in the spirit of the Lagrange multiplier rule. Consider a curve $u(t) \in \Phi^{-1}(q)$, s.t. $v = \dot{u}(0) \in \ker D_{u(0)}\Phi$. Then using the Lagrange multiplier rule, we obtain
$$
\left.\frac{d^2}{d t^2}\right|_{t=0} \varphi = \left.\frac{d}{d t}\right|_{t=0} \frac{d \varphi}{d u} \dot{u} = \left\langle\frac{d^2 \varphi}{du^2} \dot{u},\dot{u}\right\rangle + \frac{d \varphi}{d u}\ddot{u} =\left\langle \frac{d^2 \varphi}{du^2} v,v\right\rangle + p \frac{d \Phi}{d u}\ddot{u}(0).
$$
On the other hand we can twice differentiate the constraints $\Phi(u(t)) = q$. We get similarly
$$
\left\langle\frac{d^2 \Phi}{du^2} v,v\right\rangle + \frac{d \Phi}{d u}\ddot{u}(0) = 0.
$$
If we assume that $U$ is Hilbert manifold, then the two expressions give a formula for the Hessian in local coordinates that can be written as follows:
$$
\Hess_u\varphi(v,v) = \left\langle \left( \frac{d^2\varphi}{du^2} - p\frac{d^2\Phi}{du^2} \right) v,v \right\rangle := \langle Qv,v\rangle,
$$
s.t. $v$ satisfies
$$
\frac{d \Phi}{d u}v = 0.
$$

We define
$$
L = \left\{(p,q,u): \Phi(u) = q\;,\; \frac{d\varphi}{du} - p\frac{d\Phi}{d u} = 0\right\},
$$
and
$$
\cL = \pi(L),
$$
where $\pi(u,\lambda) = \lambda$. $\cL$ is the set of all Lagrangian multipliers. We say that $(\Phi,\varphi)$ is a \emph{Morse pair} (or a \emph{Morse problem}), if the equation (\ref{eq:Lag}) is regular, i.e. zero is a regular value for the map
\begin{equation}
\label{eq:map}
(p,q,u) \mapsto \frac{d\varphi}{du} - p\frac{d\Phi}{d u}.
\end{equation}
If $\dim U<\infty$, then generically constraint optimization problems are Morse. Not all functions $\varphi|_{\Phi^{-1}(q)}$ though are Morse, as one could think the name suggests. The Morse property of a constrained optimization problem implies the following important facts
\begin{proposition}
\label{prop:properties}
Let $(\varphi,\Phi)$ be a Morse problem. Then $L$ is a smooth manifold and $\pi|_{L}$ is a Lagrangian immersion into $T^*M$.
\end{proposition}
The main corollary of this proposition is that $\cL$ has a well defined tangent Lagrangian subspace at each point
\begin{equation}
\label{eq:L}
L_{(u,\lambda)}(\varphi,\Phi) = \left\{(\delta p,\delta q): \exists \delta u\in T_u U \; ; \; \frac{d\Phi}{du} \delta u = \delta q \;,\; Q\delta u = \delta p\frac{d\Phi}{du} \right\}
\end{equation}
and these subspaces will be the main objects of our study.

Before proving the last proposition, we prove a lemma
\begin{lemma}
The point $(p,q,u)$ is regular for the map (\ref{eq:map}) if and only if $\mathrm{im}\,Q$ is closed and
$$
\ker Q \cap \ker \frac{d\Phi}{d u} = 0.
$$
\end{lemma}

\begin{proof}
We compute the differential of the map (\ref{eq:map}):
$$
Q\delta u -\delta p \frac{d \Phi}{d u}.
$$
The differential is surjective if and only if it's image is closed and has a trivial orthogonal complement or, equivalently, the existence  of $w\in T_{u} U$ s.t. for any $(\delta u,\delta p)$
\begin{equation}
\label{eq:diff}
\langle Q\delta u,w\rangle -\delta p \frac{d \Phi}{d u}w = 0,
\end{equation}
 implies that $w = 0$. But since $(\delta u,\delta p)$ are arbitrary and $Q$ is symmetric, (\ref{eq:diff}) is equivalent to the existence of $w\in T_u U$, s.t. we have simultaneously
$$
Qw = 0, \qquad \frac{d\Phi}{du}w = 0.
$$
Then by assumption we have $w = 0$ and the result follows.
\end{proof}

\begin{proof}[Proof of Proposition~\ref{prop:properties}]
The fact that $L$ is a manifold is just a consequence of the implicit function theorem. We prove now that $\pi_M$ is an immersion. Differential of this map takes the tangent space to $L$ and maps it to the space $L_{(u,\lambda)}(\varphi,\Phi)$ (see (\ref{eq:L})). It ``forgets'' $\delta u$. So a non-trivial kernel of the differential must lie in the subspace $\delta p=\delta q = 0$. But from the definition of $L_{u,\lambda}(\varphi,\Phi)$ we have that in this case $Q\delta u =0$ and $D_u \Phi \delta u = 0$, which contradicts to the fact that the problem is Morse, as it can be seen from the previous lemma. Thus the differential is injective.

To prove that this immersion is Lagrangian it is enough to prove that $L_{(u,\lambda)}(\varphi,\Phi)$ is a Lagrangian subspace. Since $Q$ is symmetric, it is easy to see that this subspace is isotropic. Take $(\delta p_1,\delta q_1)$ and $(\delta p_2, \delta q_2)$ in  $L_{(u,\lambda)}(\varphi,\Phi)$. Then we compute
$$
\sigma\left( (\delta p_1,\delta q_1),(\delta p_2,\delta q_2) \right) = \delta p_1\delta q_2 - \delta p_2\delta q_1 = \delta p_1\frac{d\Phi}{du}\delta u_2 - \delta p_2\frac{d\Phi}{du}\delta u_1 = \langle Q\delta u_1,\delta u_2\rangle - \langle Q\delta u_2,\delta u_1\rangle = 0.
$$

Now it just remains to prove that the dimension of this space is equal to $n$. We are going to do it only in the finite-dimensional setting but this is true in general~\cite{ABB}. Note that if we fix $(\delta p,\delta u)$ as in the definition of $L_{(u,\lambda)}(\varphi,\Phi)$, then $\delta q$ is determined automatically. So it is enough to study the map
$$
S:(\delta u,\delta p) \mapsto Q\delta u - \delta p\frac{d\Phi}{du}.
$$
Then clearly $\dim L_{(u,\lambda)}(\varphi,\Phi) = \dim \ker S$. But we have seen in the proof of the previous lemma, that this map is actually surjective. Then $\dim \im S = \dim U$ and we have
$$
\dim \ker S = \dim( U \times \R^n) - \dim U = n.
$$
\end{proof}

\begin{figure}
\begin{center}
\includegraphics[scale=0.5]{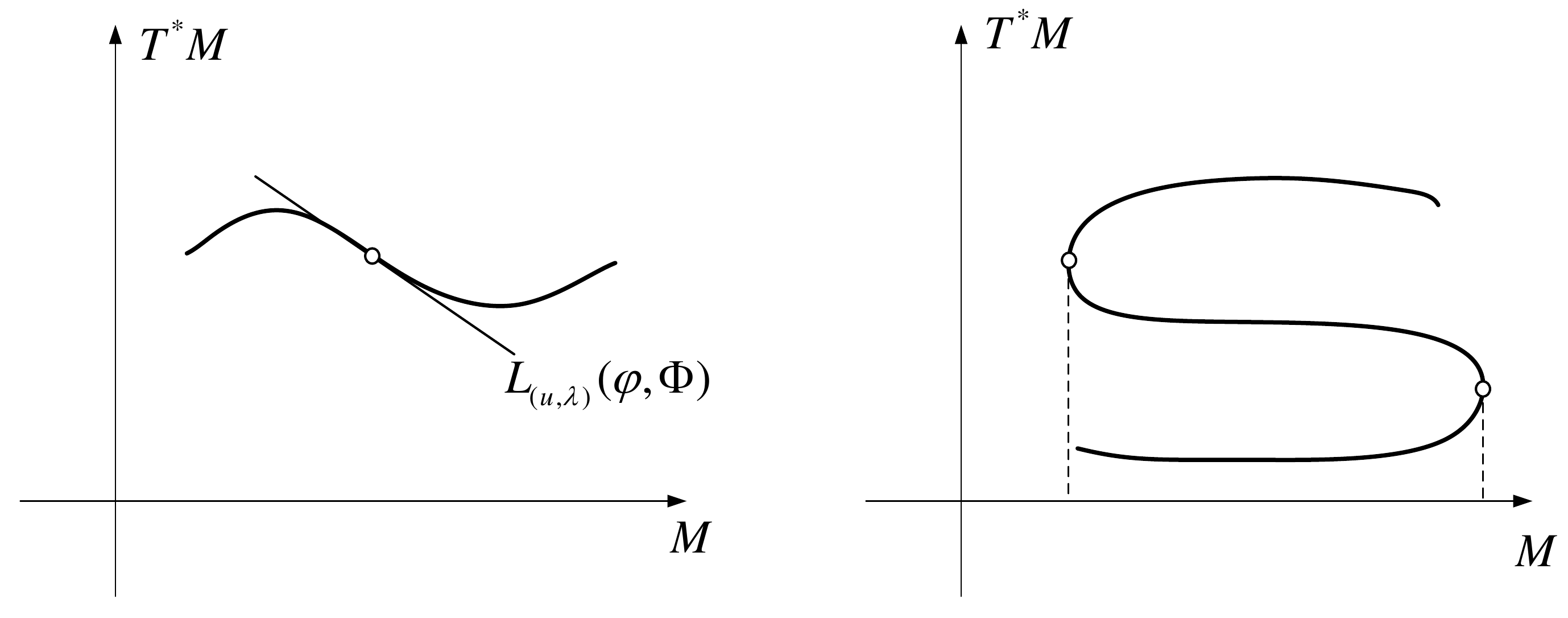}
\caption{Lagrangian manifolds in the simplest case of $\dim M = 1$}
\label{fig1}
\end{center}
\end{figure}

In a Morse problem, the Lagrange submanifold $\cL$ often contains all necessary information about the Morse index index and the nullity of the Hessian. We can already give a geometric characterization of the nullity, while for the geometric characterization of the index we will need more facts from linear symplectic geometry.
\begin{proposition}
$\Hess_u \varphi$ has a non-trivial kernel if and only if $\lambda$ is a critical point of the map $\pi_M|_\cL$, where  $\pi_M: T^*M \to M$ is the standard projection. The dimension of the kernel of the Hessian is equal to the dimension of the kernel of the differential of $\pi_M|_\cL$.
\end{proposition}

Schematically this situation is depicted in the Figure~\ref{fig1} on the right.

\begin{proof}
Note that $\lambda$ is a critical point of $\pi_M|_\cL$ if and only if the tangent space $L_{(u,\lambda)} (\varphi,\Phi)$ contains a vertical direction $(\delta p,0)$. If this is the case, by definition there exists $\delta u\in \ker D_u \Phi$, s.t.
\begin{equation}
\label{eq:cond}
Q\delta u = \delta p\frac{d\Phi}{du}.
\end{equation}
Then clearly for any $v\in \ker D_u \Phi$,  we have $\langle Q\delta u,v \rangle = 0$.

On the contrary if $D_u\Phi \delta u = 0$ and $\delta u$ belongs to the kernel of the Hessian, then for any $v\in \ker D_u \Phi$ we have $\langle Q\delta u,v \rangle = 0$ and $Q\delta u$ must be a linear combinations of the rows of $D_u \Phi$, i.e. there exists $\delta p\in T_\lambda(T^*_q M)$, s.t. (\ref{eq:cond}) holds.
\end{proof}

If the problem is Morse, then the subspace $L_{(u,\lambda)}(\varphi,\Phi)$  defined in (\ref{eq:L}) is Lagrangian. However, our goal is to handle degenerate cases and a starting point is the following surprising fact that is valid without any regularity assumption on $(\varphi,\Phi)$; in particular, $u$ may be a critical point of $\Phi$.

\begin{proposition}
If $U$ is finite dimensional, then the defined in (\ref{eq:L}) space $L_{(u,\lambda)}(\varphi,\Phi)$ is a Lagrangian subspace of $T_{\lambda}(T^*M)$.

\end{proposition}

We leave the proof of this interesting linear algebra exercise to the reader. Next example shows that finite dimensionality of $U$ is essential.

Let $M = \R$ and $U$ be a Hilbert space. Then $a = \Phi'_u$ is an element of the dual space $U^*,\ T^*M=\mathbb R^2$, and Lagrangian subspaces are just one-dimensional subspaces of $\mathbb R^2$. We set
 $L=L_{(u,\lambda)}(\varphi,\Phi)$. By the definition we have
$$
L = \{(\delta p, \delta q): \exists \,v\in U\, ,\, Q v = (\delta p) a, \delta q = \langle a, v\rangle\}
$$
Assume that $Q:U\to U^*$ is injective and not surjective. If $a\notin \im Q$, then $L = \{(0,0)\}$ with a unique lift $v = 0$.

Injectivity and symmetricity of $Q$ imply that $\im Q$ is everywhere dense in $U$, hence $\mathrm{im}Q$ is not closed in the just described example.
Now we drop the injectivity assumption but assume that $\im Q = \overline{\im Q}$. Let us show that $L$ is 1-dimensional in this case. Indeed, the self-adjointness of $Q$ implies that $\ker Q \oplus \im Q = U$. Then we have two possible situations
\begin{enumerate}
\item $a\in \im Q$. Then there is a unique preimage of $a$ in $\im Q$, that we denote by $v = Q^{-1}a$. We get
$$
L = span\{(1, \langle a, Q^{-1} a \rangle )\},
$$
where $Q^{-1}$ is a pseudo-inverse.
\item $a\notin \im Q$, then we must have $\delta p = 0$ and $v \in \ker Q$. There exists $v\in\ker Q$ such that $\langle a,v\rangle\ne 0$, and we obtain
$$
L = span\{(0,\langle a ,v \rangle)\}.
$$
\end{enumerate}

If $\dim U<\infty$, then $\im Q$ is automatically closed and $L$ is Lagrangian as we have seen.

\subsection{Lagrangian Grassmanian and Maslov index}
\label{sec:lagr_gr}

We are going to give a geometric interpretation of the Morse index of the Hessian in terms of some curves of Lagrangian subspaces. To do this, we need some results about the geometry of the set of all Lagrangian subspaces of a given symplectic space $(W,\sigma)$. This set has a structure of a smooth manifold and is called the \emph{Lagrangian Grassmanian} $L(W)$. We give just the basic facts about $L(W)$. For more information see~\cite{gosson}.

To construct a chart of this manifold we fix a Lagrangian subspace $\Lambda_2 \in L(W)$ and consider the set of all Lagrangian subspaces transversal to $\Lambda_2$, which we denote by $\Lambda_2^\pitchfork$ (the symbol $\pitchfork$ means "transversal"). By applying a Gram-Schmidt like procedure that involves $\sigma$, we can find some Darboux coordinates $(p,q)$ on $W$, s.t. $\Lambda_2 = \{(0,q)\}$ (see~\cite{ABB} for details). Then $\Lambda_0 = \{(p,0)\}$ belongs to $\Lambda_2^\pitchfork$ and any other $\Lambda_1\in \Lambda_2^\pitchfork$ can be defined as a graph of a linear map from $\Lambda_0 $ to $\Lambda_2$. As we have seen in the Section~\ref{sec:bas_symp} the matrix of this map is symmetric and we obtain the identification of $\Lambda_2^\pitchfork$ with the space of symmetric $n\times n$-matrices that gives the desired local coordinates on $L(W)$.

\begin{figure}[h]
\begin{center}
\includegraphics[scale=0.7]{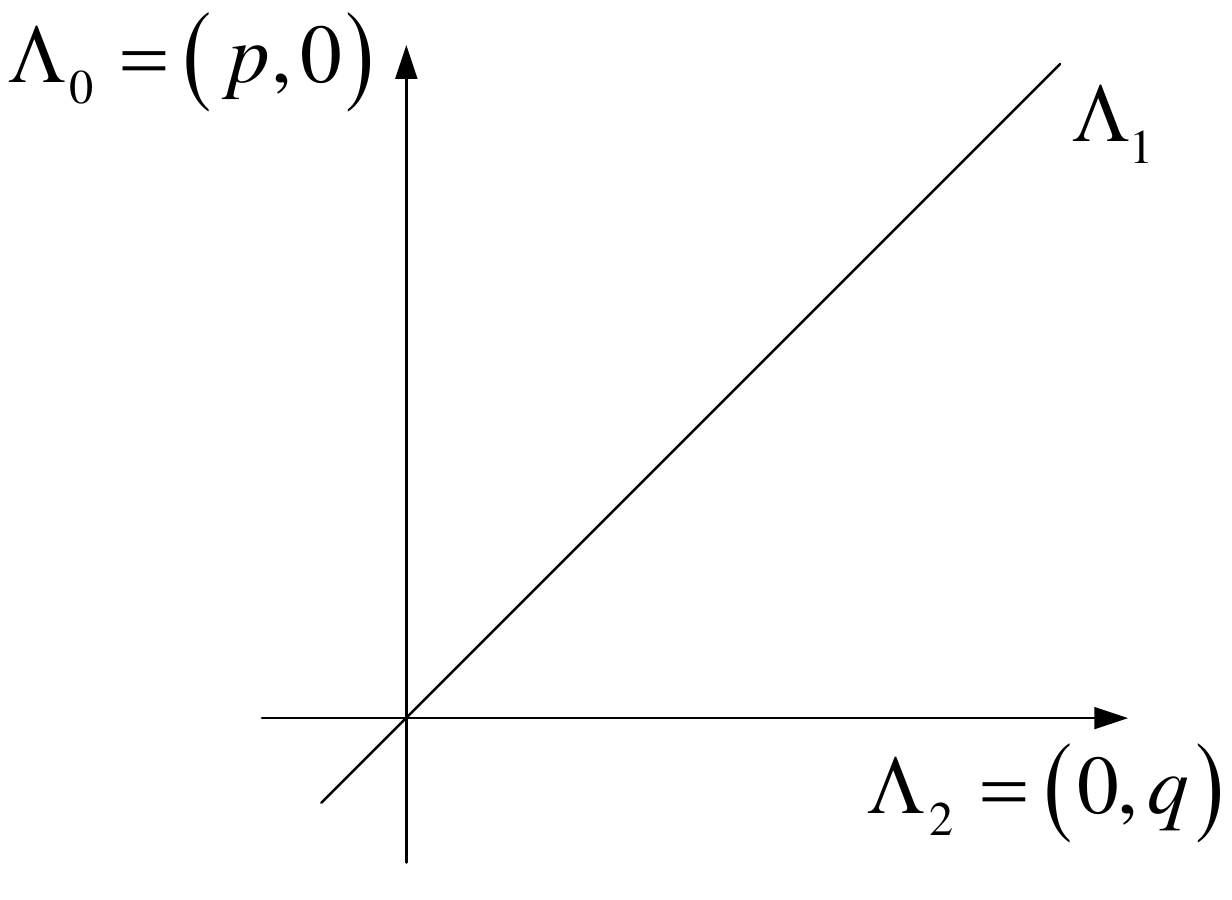}
\end{center}
\caption{Defining a Lagrangian plane $\Lambda_1$ as a graph of a quadratic form from $\Lambda_0$ to $\Lambda_2$}
\end{figure}

Since symplectic maps preserve the symplectic form, they also map Lagrangian subspaces to Lagrangian subspaces. This action is transitive, i.e. there are no invariants of the symplectic group acting on $L(W)$. If we consider the action of the symplectic group on pairs of Lagrangian spaces, then the only invariant is the dimension of their intersection~\cite{gosson}. But triples $(\Lambda_0,\Lambda_1,\Lambda_2)$ do have a non-trivial invariant, that is called the \emph{Maslov index} of the triple or the \emph{Kashiwara index}.

To define it suppose that $\Lambda_0,\Lambda_1,\Lambda_2 \in L(W)$, s.t. $\Lambda_0 \pitchfork \Lambda_1$ and $\Lambda_0 \pitchfork \Lambda_2$. Since the symplectic group acts transitively on the space of pairs of transversal Lagrangian planes, we can assume without any loss of generality that $\Lambda_0 = \{(p,0)\}$ and $\Lambda_2 = \{(0,q)\}$. Then we can identify $\Lambda_1$ with a graph $\{(p,Sp)\}$, where $S$ is a symmetric matrix. The invariant $\mu$ of this triple is then defined as
$$
\mu(\Lambda_0,\Lambda_1,\Lambda_2) = \sgn S.
$$

This invariant can be also defined intrinsically as the signature of a quadratic form $\tilde{q}: \Lambda_1\times \Lambda_1 \to R$, that is defined as follows. Since $\Lambda_0 \pitchfork \Lambda_2$, any $\lambda \in \Lambda_1$ can be decomposed as $\lambda= \lambda_0 + \lambda_2$, where $\lambda_i \in \Lambda_i$. Then we set $\tilde{q}(\lambda) = \sigma (\lambda_0,\lambda_2)$. One can check that those definitions agree.

This invariant has a couple of useful algebraic properties. The simplest ones is the antisymmetry:
$$
\mu(\Lambda_2,\Lambda_1, \Lambda_0) = - \mu(\Lambda_0,\Lambda_1, \Lambda_2);
$$
$$
\mu(\Lambda_0,\Lambda_2, \Lambda_1) = -\sgn (S^{-1})= - \mu(\Lambda_0,\Lambda_1, \Lambda_2)
$$
We are going to state and prove another important property called the chain rule, after we look more carefully at the geometry of $L(W)$. Let us fix some $\Delta \in L(W)$. The \emph{Maslov train} $\cM_\Delta$ is the set $\cM_\Delta = L(W)\setminus \Delta^\pitchfork$ of Lagrangian planes that have a non-trivial intersection with $\Delta$. It is an algebraic hyper-surface with singularities and its intersection with a coordinate chart containing $\Delta$ can be identified with the set of degenerate symmetric matrices.

The set of nonsmooth points of the hyper-surface $\cM_\Delta$ consist of Lagrangian subspaces that have an intersection with $\Delta$ of dimension two or more. It easy to check that this singular part has codimension two in $\cM_\Delta$. It follows that the intersection number mod 2 of $\cM_\Delta$ with any continuous curve whose endpoints do not belong to $\cM_\Delta$ is well defined and is homotopy invariant. For example when $\dim W = 4$, we have that the intersection of $\cM_\Delta$ with a coordinate chart is identified with $2\times 2$ symmetric matrices with zero determinant. This is a cone whose points except the origin correspond to Lagrangian planes that have a one-dimensional intersection with $\Delta$. The origin represents the dimension two intersection with $\Delta$, which is equal to $\Delta$ itself in this case. Clearly a general position curve in $L(W)$ does not intersect the origin (see Figure~\ref{fig3}) as well as a general position homotopy of curves, and so the intersection number mod 2 is well defined.

\begin{figure}[h]
\begin{center}
\includegraphics[scale=1]{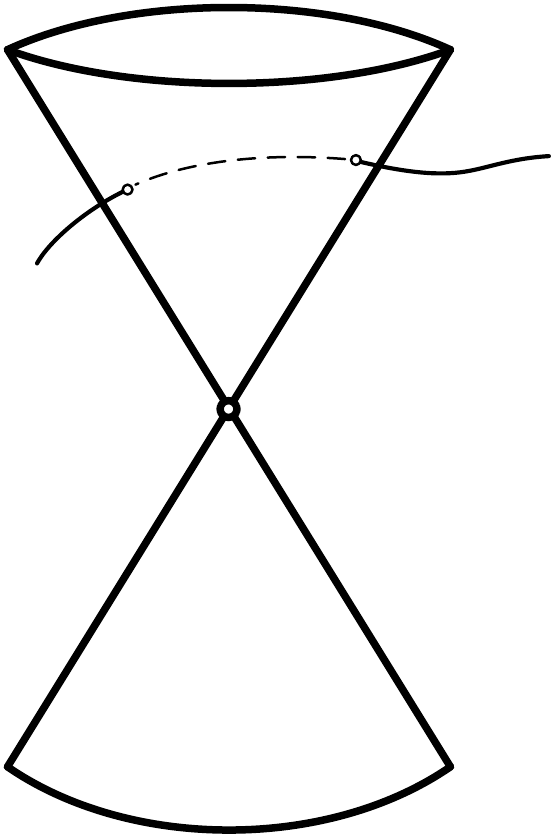}
\end{center}
\caption{A curve in a local chart of the Lagrangian Grassmanian $L(\R^4)$ and the
Maslov train}
\label{fig3}
\end{figure}

We would like to define the integer-valued intersection number of curves in $L(W)$ with $\cM_\Delta$, and for this we need a coorientation of the smooth part of $\cM_\Delta$. Let $\Lambda_t\in L(W)$ s.t. $\Lambda_0 = \Lambda$ and consider any $\lambda_t \in \Lambda_t$ s.t. $\lambda_0 = \lambda$. Then we define
$$
\underline{\dot{\Lambda}}(\lambda) = \sigma(\lambda,\dot{\lambda}).
$$
Thus we see that to any tangent vector $\dot\Lambda$ we can associate a quadratic form $\underline{\dot{\Lambda}}$. Is easy to see that $\underline{\dot{\Lambda}}(\lambda)$ is indeed a well-defined quadratic form, i.e. that $\sigma(\lambda,\dot{\lambda})$ depends only on $\dot\Lambda$ and $\lambda$. Moreover, $\dot\Lambda\mapsto \underline{\dot{\Lambda}}(\lambda),\ \dot\Lambda\in T_\Lambda L(W)$ is an isomorphism of $T_\Lambda L(W)$ on the space of quadratic forms on $\Lambda$.

From the previous discussion we know that $\dim(\Lambda_t \cap \Delta)=1$ if $\Lambda_t$ is a smooth point of $\cM_\Delta$. Let $\lambda \in \Lambda_t \cap \Delta,\ \lambda\ne 0$; the intersection is transversal at the point $\Lambda_t$ if and only if $\underline{\dot{\Lambda}}(\lambda)\ne 0$. We say that the sign of the intersection is positive, if $\underline{\dot{\Lambda}}(\lambda) > 0$ and negative otherwise. The intersection number of a continuous curve $t\mapsto\Lambda_t$ with $\cM_\Delta$ is called the Maslov index of the curve with respect to $\cM_\Delta$. Note that Maslov index of a closed curve (i.e. a curve without endpoint) does not depend on the choice of $\Delta$. Indeed, the train $\cM_\Delta$ can be transformed to any other train by a continuous one-parametric family of symplectic transformations and Maslov index is a homotopy invariant.

The Maslov index of a curve and the Maslov index of a triple are closely related. Let $\gamma: [0,1] \to L(W)$, s.t. the whole curve does not leave the chart $\Delta^\pitchfork$. Then we have
\begin{equation}
\label{eq:maslov_formula}
2 \gamma \circ \cM_\Lambda = \mu(\Lambda,\gamma(1),\Delta)-\mu(\Lambda,\gamma(0),\Delta)
\end{equation}
what easily follows from definitions.

Now we can state the last property.
\begin{lemma}[The chain rule]
Let $\Lambda_i \in L(W)$, $i=0,1,2,3$. Then
$$
\mu(\Lambda_0,\Lambda_1,\Lambda_2)+\mu(\Lambda_1,\Lambda_2,\Lambda_3)+\mu(\Lambda_2,\Lambda_3,\Lambda_0)+\mu(\Lambda_3,\Lambda_0,\Lambda_1) = 0.
$$
\end{lemma}
\begin{proof}
Connect $\Lambda_0, \Lambda_2$ with two curves: one that is completely in $\Lambda_3^\pitchfork$ and another one that is completely in $\Lambda_1^\pitchfork$. Schematically this situation is depicted in Figure~\ref{fig4}. This gives a closed curve $\gamma$ in $L(W)$ and we can compute it's intersection with $\cM_{\Lambda_1}$ and $\cM_{\Lambda_3}$.
The chain rule follows from the identities
$ \gamma \circ \cM_{\Lambda_1} = \gamma \circ \cM_{\Lambda_3}$ and (\ref{eq:maslov_formula}).
\end{proof}

\begin{figure}[h]
\begin{center}
\includegraphics[scale=1]{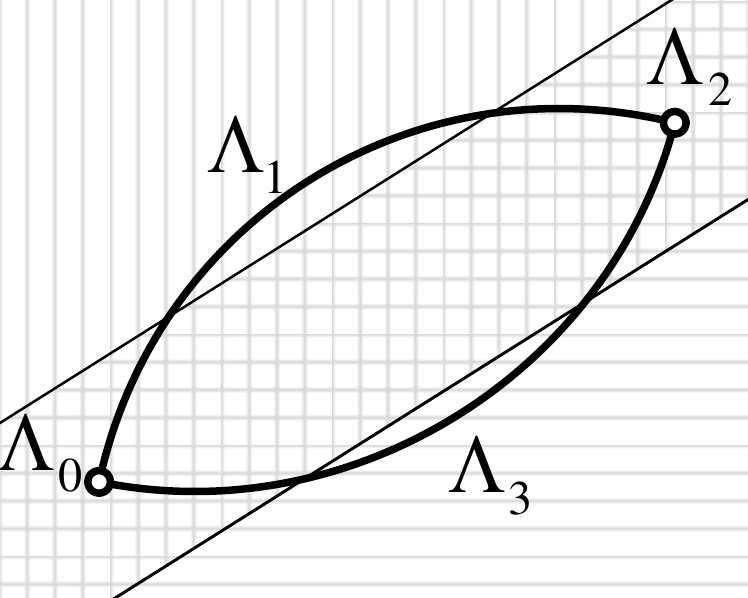}
\caption{An illustration to the proof. $\Lambda_3^\pitchfork$ correspond to the region with vertical lines and $\Lambda_1^\pitchfork$ to the region with the horizontal ones}
\label{fig4}
\end{center}

\end{figure}

The formula (\ref{eq:maslov_formula}) allows us to compute the Maslov index of a continuous curve, without putting it in general position and really computing the intersection. We just need to split the whole curve into small pieces, s.t. that each of them lies in a single coordinate chart, and then compute the index of the corresponding triples. This motivates the following definition. A curve $\gamma(t) \in L(W)$ is called \emph{simple}, if there exists $\Delta\in L(W)$, s.t. $\gamma(t)\in \Delta^\pitchfork$.

\section{Lecture 2}

\subsection{Morse Index}
\label{sec:morse}

Now we return to the study of the second variation for a finite-dimensional Morse problem. We have seen that $\cL$ is an immersed Lagrangian submanifold. This allows us to define the Maslov cocycle in the following way. We fix a curve $\gamma: [t_0,t_1]\to \cL$ and consider the corresponding curve $L_t: t\mapsto T_{\gamma(t)}\cL$. Along $\gamma(t)$ we have vertical subspaces $\Pi = T_{\gamma(t)}(T^*_{q(t)}M)$ which lie in different symplectic spaces. Vector bundles over the segment $[t_0,t_1]$ are trivial, so by using a homotopy argument, we can assume that the symplectic space and the vertical subspace $\Pi$ are fixed. We define the \emph{Maslov cocycle} as
$$
\mu(\gamma(t)) = L_t \circ \cM_\Pi.
$$

Then we have the following theorem.
\begin{theorem}
Let $\gamma(t)$ be a curve connecting $q_{t_0}$ with $q_{t_1}$. Then
$$
-2\mu(\gamma) = \delta \left( \sgn \Hess_u \varphi|_{\Phi^{-1}}(q) \right) := \sgn \Hess_{u_{t_1}} \varphi|_{\Phi^{-1}(q_{t_1})} - \sgn \Hess_{u_{t_0}} \varphi|_{\Phi^{-1}(q_{t_0})}
$$
\end{theorem}


\begin{proof}[Sketch of the proof]
 We denote $\overline{Q}(v) = \langle Qv,v\rangle$. Then
$$
\Hess_u\varphi|_{\Phi^{-1}(q)}= \overline{Q}|_{\ker \Phi'_u}, \qquad \Phi(u) = q
$$
and the difference of the signatures will be the difference of the correspondent signatures of $\overline{Q}|_{\ker \Phi'_u}$. Now we do not restrict $Q$ to the kernel of $\Phi'_u$ and use the fact that it depends on the choice of coordinates. Indeed, the vertical subspace is fixed, but we have a freedom in choosing the horizontal space.

\textbf{Exercise:} Given a Lagrangian subspace $\Lambda\subset T_\lambda(T^*M)$ that is transversal to the fiber $T^*_qM$, there exist local coordinates in which $\Lambda=\{(0,x): x\in\mathbb R^n\}$. Operator $Q$ is non-degenerate iff the horizontal subspace $\{(0,x): x\in\mathbb R^n\}$ is transversal to $L_{(u,\lambda)}(\varphi,\Phi)$.

\smallskip
To proof the theorem, we may divide the curve into small pieces and check the identity separately for each piece.
In other words, we can assume that the curve is contained in the given coordinate chart and, according to the exercise, that $Q_t$ is not degenerate for all $t\in[t_0,t_1]$.

Then we can apply the following linear algebra lemma.
\begin{lemma}
\label{lem:important}
Let $E$ be a possibly infinite-dimensional Hilbert space, $Q$ a quadratic form on $E$ that is positive definite of a finite codimension subspace and $V$ a closed subspace of $E$. Then if we denote by $V^\perp_Q$ the orthogonal complement of $V$ in $E$ w.r.t. $Q$, the following formula is valid
$$
\ind^- Q = \ind^- Q|_V + \ind^- Q_{V^\perp_Q} + \dim(V\cap V^\perp_Q) - \dim(V\cap \ker Q).
$$
Moreover if $Q|_V$ is non degenerate, then $\dim(V\cap V^\perp_Q) = \dim(V\cap \ker Q) = 0$.
\end{lemma}
Thus by construction we get
$$
\sgn \overline{Q}_{t_i} =  \sgn \overline{Q}_{t_i}|_{\ker \Phi'_u} + \sgn \overline{Q}_{t_i}|_{(\ker \Phi'_u)^\perp_{Q_{t_i}}}.
$$
Since $Q_t$ is nondegenerate and continuously depends on $t$,  we have $\sgn Q_{t_1} = \sgn Q_{t_0}$. Then first summand is just the Hessian and one can show that
$$
\sgn\overline{Q}_{t_i}|_{(\ker \Phi'_u)^\perp_{Q_{t_i}}} = \sgn S_{t_i}=\mu\bigr(Ver,L_t,Hor\bigr),
$$
where $Ver=\{(\xi,0):\xi\in\mathbb R^n\},\ Hor=\{(0,x):x\in\mathbb R^n\}.$

The statement of the theorem now follows from (\ref{eq:maslov_formula}).
\end{proof}

What about the infinite-dimensional Morse problem? The signature of the Hessian is not defined in this case but difference of the signatures can be substituted by the difference of the Morse indices if $\Hess_u\varphi|_{\Phi^{-1}(q)}$ are positive definite on a finite codimension subspace and $\Hess_{u_{t_i}} \varphi|_{\Phi^{-1}(q_{t_i})},\ i=0,1,$ are nondegenerate.

\subsection{General case}
\label{sec:general}

Consider now a general, not necessary a Morse constrained optimization problem $(\varphi,\Phi)$ and a couple $(u,\lambda)$ that satisfies the Lagrange multiplier rule $d_u\phi=\lambda D_u\Phi$. In local coordinates:
$$
\lambda=(p,q);\qquad\varphi'_u = p \Phi'_u, \quad \Phi(u) = q.
$$
We would like to consider the subspace $L_{(u,\lambda)}(\varphi,\Phi)$ defined in~\ref{sec:sec_var}, but in general it is just an isotropic subspace. Nevertheless if $V \in \cU$ is  finite dimensional, then the space $L$ corresponding to $(\varphi,\Phi)|_V$ is Lagrangian and we denote it by $L_{(u,\lambda)}(\varphi,\Phi)|_V$.

The set of all finite-dimensional subspaces has a partial ordering given by inclusion. Moreover it is a directed set, therefore we can take a generalized limit over the sequence of nested subspaces. The existence of this limit is guaranteed by the following

\begin{theorem}[\cite{agr_feedback}]
The limit
$$
\fL_{(u,\lambda)}(\varphi,\Phi) = \lim_{V\nearrow U} L_{(u,\lambda)}(\varphi,\Phi)|_V
$$
exists if and only if $\ind^- Q|_{\Phi'_u}<\infty$.
\end{theorem}

If the limit exists we call it the \emph{$L$-derivative} and denote it by the gothic symbol $\fL$ to distinguish it from the isotropic subspace that we would have got otherwise. The $L$-derivative constructed over some finite-dimensional subspace of the source space we will call a \emph{$L$-prederivative}. From here we also omit for brevity $(u,\lambda)$ in the notations. The following property allows to find efficient ways to compute $\fL(\varphi,\Phi)$.
\begin{theorem}
Suppose that $U$ is a topological vector space, s.t. $\varphi,\Phi$ are continuous on $U$ and $U_0 \subset U$ is a dense subspace. Then
$$
\fL_{(u,\lambda)}(\varphi,\Phi)|_{U_0} = \fL_{(u,\lambda)}(\varphi,\Phi)|_{U}
$$
\end{theorem}

One can use this theorem in two different directions. Given a topology on $U_0$ one can look for a weaker topology on $U_0$, s.t. $\varphi$ and $\Phi$ are continuous in that topology. Then we extend $U_0$ to $U$ by completion. This trick was previously used in~\cite{agr_feedback}.

Another way is to take a smaller subspace. For example, if $U$ is separable, then we can take a dense countable subset $e_1,e_2,...$ and compute the limit as
$$
\lim_{n\to\infty}L(\varphi,\Phi)|_{span\{e_1,...,e_n\}} = \fL(\varphi,\Phi).
$$

This allows to see how the $L$-derivative changes as we add variations. Under some additional assumptions we can also compute the change in the Maslov index after adding additional subspaces (see the Appendix).

\subsection{Monotonicity}
\label{sec:monot}

To discuss the Morse theorem in the general setting we need a very useful notion of monotonicity. Assume that the curve $\gamma_t$ is contained in a chart $\Delta^\pitchfork$, then one can associate a one parametric family of quadratic forms $S_t$, where $\gamma_t=\{(p,S_tp):p\in\mathbb R^n\},\ \Delta=\{(0,q):q\in\mathbb R^n\}$. We say that $\gamma_t$ is \emph{increasing} if $S_t$ is increasing, i.e. $S_t - S_\tau$ is positive definite, when $t > \tau$. It is important that the property of a smooth curve to be increasing does not depend on the choice of a coordinate chart. Indeed, the quadratic form $p\mapsto\langle\dot S_tp,p\rangle$ is equivalent by a linear change of variables to the form $\underline{\dot\gamma}_t$ defined in Section~\ref{sec:lagr_gr}, and the definition of $\underline{\dot\gamma}_t$ is intrinsic. It does not use local coordinates.

Moreover, if $\gamma_t,\ t_0\le t\le t_1$ is simple and increasing, then Maslov index $\gamma \circ \cM_{\Lambda}$ depends only on $\gamma_{t_0},\gamma_{t_1},\Lambda$ and can be explicitly expressed via the Maslov index of this triple. More precisely, assume that $\gamma_{t_0},\gamma_{t_1},\Lambda$ are mutually transversal and let $\tilde q$ be a quadratic form on $\lambda$ defined by the formula: $\tilde q(\lambda)=\sigma(\lambda_1,\lambda_0),\ \lambda\in\Lambda$, where $\lambda_0\in\gamma_{t_0},\lambda_1\in\gamma_{t_1},\lambda=\lambda_0+\lambda_1$  Actually if we define
$$
\Ind_{\Lambda}(\gamma_{t_0},\gamma_{t_1}) = \ind^- \tilde{q},
$$
then one can show~\cite{agrachev} that
$$
\gamma \circ \cM_{\Lambda} = \Ind_{\Lambda}(\gamma_{t_0},\gamma_{t_1}).
$$
A corollary of this fact is the following triangle inequality:
\begin{proposition}
Let $\Pi, \Lambda_i$, $i=0,1,2$ be Lagrangian subspaces in $L(\Sigma)$. Then
$$
\Ind_{\Pi}(\Lambda_0,\Lambda_2) \leq \Ind_{\Pi}(\Lambda_0,\Lambda_1) + \Ind_{\Pi}(\Lambda_1,\Lambda_2)
$$
\end{proposition}

\begin{proof}

We consecutively connect $\Lambda_0$ with $\Lambda_1$, $\Lambda_2$ with $\Lambda_3$, and $\Lambda_3$ with $\Lambda_0$ by simple monotone curves that gives us a closed curve $\gamma$. Then we have
$$
\gamma \circ \cM_\Pi = \Ind_\Pi(\Lambda_0,\Lambda_1) + \Ind_\Pi(\Lambda_1,\Lambda_2) + \Ind_\Pi(\Lambda_2,\Lambda_0).
$$
From the definition of $\Ind$, one has
$$
\Ind_\Pi(\Lambda_2,\Lambda_0) = n - \Ind_\Pi(\Lambda_0,\Lambda_2).
$$
So it is enough to show that $\gamma\circ \cM_\Pi \geq n$. This again follows from the fact that the intersection index of a closed curve does not depend on the choice of $\Pi$.
Recall that the group of symplectic transformations acts transitively on the set of pairs of transversal Lagrangian planes. Hence we can find $\Delta \in L(\Sigma)$ s.t. such that $\Lambda_0$ and $\Lambda_1$ belong to the coordinate chart
$\Delta^\pitchfork$ and, moreover, $\Lambda_0$ is represented by a negative definite symmetric matrix in this chart while  $\Lambda_1 $ is represented by a positive definite symmetric matrix. Then
$$
\gamma \circ \cM_\Pi = \gamma \circ \cM_\Delta = \Ind_\Delta(\Lambda_0,\Lambda_1) + \Ind_\Delta(\Lambda_1,\Lambda_2) + \Ind_\Delta(\Lambda_2,\Lambda_0) \geq n
$$
since by definition $\Ind_\Delta(\Lambda_0,\Lambda_1) = n$ and $\Ind_\Delta(\Lambda_i,\Lambda_j)\geq 0$.
\end{proof}

So if we take a curve $\gamma(t) \in L(\Sigma)$ and it's subdivision at moments of time $0 = t_0 < t_1 < ... < t_N = 1$, we can consider the sum
$$
\sum_{i=0}^{N-1} \Ind_\Pi(\gamma(t_{i}),\gamma(t_{i+1}) )
$$
which grows as the partition gets finer and finer. For monotone increasing curves this sum will stabilize and will be equal to a finite number. This motivates the next definition. A (maybe discontinuous) curve $\gamma: [0,1]\to L(\Sigma)$ is called \emph{monotone increasing} if
$$
\sup_D \sum_{t_i \in D} \Ind_\Pi(\gamma(t_{i}),\gamma(t_{i+1}) ) < \infty,
$$
where the supremum is taken over all possible finite partitions $D$ of the interval $[0,1]$.

Monotone curves have properties similar to monotone functions. For example, they have only jump discontinuous and are almost everywhere differentiable.

It is instructive to see how monotonicity works on curves in $L(\R^2)$. The Lagrangian Grassmanian $L(\R^2)$ topologically is just an oriented circle and a curve is monotone increasing if it runs in the counter-clockwise direction. A coordinate chart is the circle with a removed point. On the left picture of Figure~\ref{fig5} the blue curve is monotone and it's index Maslov is equal to zero. On the right the black curve is not monotone increasing. Indeed if we take any two points the Maslov index of a triple $(\Lambda_{t_{i}},\Pi,\Lambda_{t_{i+1}})$ will be equal to one (the red curve).

\begin{figure}
\begin{center}
\includegraphics[scale=1]{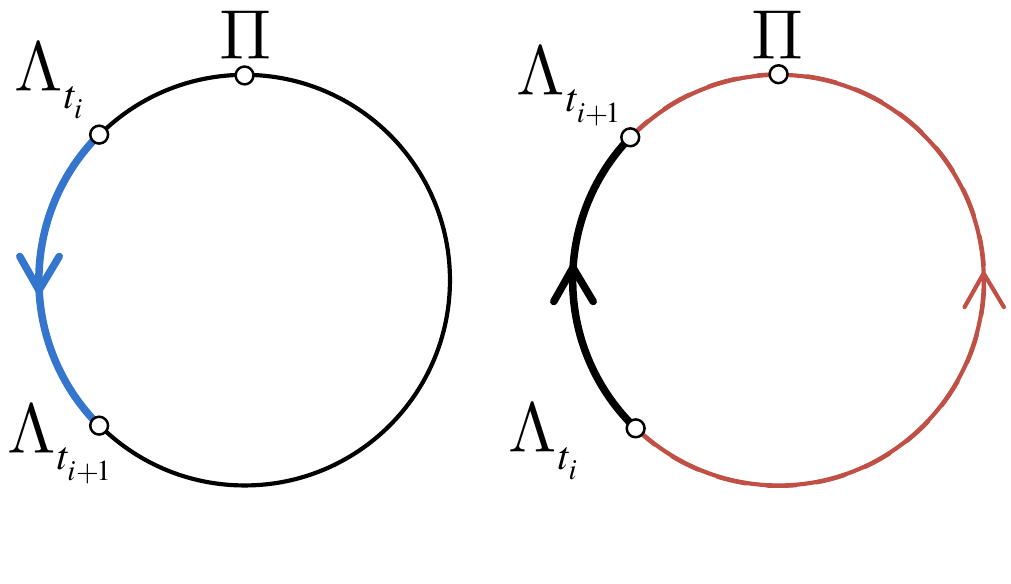}
\end{center}
\caption{Monotone increasing blue curve and monotone decreasing black curve with the same end-points}
\label{fig5}
\end{figure}

\subsection{$L$-derivative for optimal control problems with a free control}
\label{sec:l-deriv}
In this subsection we consider an optimal control problem with fixed end-points\footnote{We do not study free endpoint problems here: see recent paper \cite{Fra} and references therein for 2nd order optimality conditions in the free endpoint case.}. 
\begin{equation}
\label{eq:control}
\dot{q} = f(q,u(t)), \qquad q \in M, \qquad u\in \R^k,
\end{equation}
$$
q(0) = q_0, \qquad q(t_1) = q_1,
$$
$$
J^{t_1}_{0} = \int_{0}^{t_1} l(q(s),u(s))ds
$$

We may assume, that time moment $t_1$ is fixed. Otherwise, we can make a time scaling $s = \alpha\tau$, where $\alpha > 0$ is a constant that will be used as an additional control function. If we denote by $\hat{u}(\tau) = u(\alpha\tau)$ and $\tau_1 = t_1/\alpha$, then we get an equivalent optimal control problem
$$
\dot{q} = \alpha f(q,\hat u(\tau)), \qquad (\alpha,\hat{u}(\tau))\in \R \times L^2_k[0,\tau_1],
$$
$$
q(0) = q_0, \qquad q(\tau_1) = q_1,
$$
$$
\hat J^{\tau_1}_{0} = \alpha \int_{0}^{\tau_1} l(q(\tau),\hat u(\tau))d\tau
$$
with fixed time $\tau_1$. We see that by variating $\alpha$ we variate the final time $t_1$, so we can assume from the beginning that $t_1$ is fixed.

A curve $q(t)$ that satisfies (\ref{eq:control}) for some locally bounded measurable function $u(t)$ is called an \emph{admissible trajectory}. As in the case of the classical calculus of variations we can define the evaluation map $F_t$ that takes an admissible curve and maps it to the corresponding point at a moment of time $t$. We easily recover the Lagrange multiplier rule (\ref{eq:Lag})
\begin{equation}
\label{eq:nu}
\nu d_\gamma J_{0}^t = \lambda_t D_\gamma F_t - \lambda_{0} D_\gamma F_{0}.
\end{equation}
where $\nu$ can be normalized in such a way that it takes value $1$ or $0$. If $\nu = 1$, we call the corresponding extremal \emph{normal}, otherwise it is called \emph{abnormal}. In the case of calculus of variations one has only normal extremals because $(F_t,F_{0})$ is a submersion.

We can derive the Hamiltonian system also in this case. Locally we assume that $\lambda_t = (p_t,q_t)$. Then
$$
\nu\int_{0}^t\left( \frac{\p l}{\p q} dq_s + \frac{\p l}{\p u} du_s \right)ds = p_t dq_t - p_{0} dq_{0}.
$$
We differentiate this expression w.r.t. time $t$:
$$
\nu\frac{\p l}{\p q} dq_t + \nu\frac{\p l}{\p u} du_t = \dot p_t dq_t + p_t df,
$$
where $df=\frac{\p f}{\p q}dq_t+\frac{\p f}{\p u}du_t$.
Now we collect terms and obtain:

$$
\frac{\p }{\p u}  \left( \langle p_t, f(q_t,u)\rangle - \nu l(q_t,u) \right) = 0,
$$
$$
\dot p_t = \frac{\p}{\p q_t} \left( \nu l(q_t,u) - \langle p_t, f(q_t,u)\rangle\right),
$$
$$
\dot q= f(q,u).
$$
Thus if we set
$$
H(p,q,u) = \langle p, f(q,u)\rangle - \nu l(q,u)
$$
we see that the equations above are equivalent to a Hamiltonian system, where the Hamiltonian satisfies
$$
\frac{\p H}{\p u} = 0.
$$
We can rewrite these equations in the coordinate free form:
\begin{equation}
\label{eq:hamilton}
\dot\lambda_t=\vec H(\lambda_t,u), \qquad \frac{\p H}{\p u}=0.
\end{equation}

Now fix $q_0\in M$; then $F_{0}^{-1}(q_0)$ is just the space of control functions $u(\cdot)$. Let $E_t:u(\cdot)\mapsto q(t)$ be the {\it endpoint map}, $E_t=F_t\bigr|_{F_{0}^{-1}(q_0)}$. It is easy to see that the relation (\ref{eq:nu}) is equivalent to the relation
$$
\nu d_uJ_{0}^t=\lambda_tD_uE_t.
$$
Let $\tilde u(\cdot)$ satisfy this relation, i.e. there exist $\lambda_t,\ 0\le t\le t_1$, such that (\ref{eq:hamilton}) is satisfied with $u = \tilde{u}(t)$.
We are going to compute the $L$-derivative $\fL_{(\tilde u,\lambda_t)}(\nu J_{0}^t,E_t)$. Let
$P_{t}:M\to M$ be the flow generated by the differential equation $\dot q=f(q,\tilde u(t)),\ P_{0}=Id$. We set $G_t=P_{t}^{-1}\circ E_t$. The map $G_t$ is obtained from $E_t$ by a ``change of variables'' in $M$.
Intrinsic nature of the $L$-derivative now implies that
$$
\fL_{(\tilde u,\lambda_t)}(\nu J_{0}^t,G_t) = P_{t}^* \left( \fL_{(\tilde u,\lambda_{0})}(\nu J_{0}^t,E_t)\right)
$$
and we may focus on the computation of the $L$-derivatives
$\fL_{(\tilde u,\lambda_{0})}(\nu J_{0}^t,G_t)\subset T_{\lambda_{0}}(T^*M),\ 0\le t\le t_1$ which is more convenient, since all the $L$-derivatives lie in the same symplectic space $T_{\lambda_{0}}(T^*M)$ and this way we don't need a connection or a homotopy argument to compute the Maslov index.

To find $G_t$ we apply a time-dependent change of variables
$$
y = P_{t}^{-1}(q).
$$
It is easy to see that we get an equivalent control system
$$
\dot{y} = (P_{t}^{-1})_* \left( f(u,\cdot) - f(\tilde{u}(t),\cdot) \right)(y) = g(t,u,y),
$$
where in the center we have the pull-back of vector fields. Then the definitions give us
$$
y_0= q_0, \qquad g(\tilde{u}(t),y) = 0.
$$
Similarly in the functional we define $\psi(t,u,y) = l(P_{t}(y),u) $.

To write down an explicit expression for  $\fL_{(\tilde u,\lambda_{0})}(\nu J_{0}^t,G_t)$ we must characterize first all the critical points and Lagrange multipliers in these new coordinates. We apply the Lagrange multiplier rule exactly as above and find, that if $(y(t),u(t))$ is a critical point, then there exists a curve of covectors $\mu_t \in T_{y(t)}^*M$, that satisfies a Hamiltonian system
\begin{equation}
\label{eq:hamilt}
\dot \mu_t = \vec{h} (t,u,\mu_t),
\end{equation}
where the Hamiltonian
$$
h(t,u,\mu_t) = \langle \mu_t, g(t,u,y) \rangle - \nu \psi(t,u,y)
$$
satisfies also
\begin{equation}
\label{eq:maxim}
\frac{\p h(t,u,\mu_t)}{\p u} = 0.
\end{equation}
Note that for the referenced critical point $(y_0,\tilde{u}(t))$, the corresponding curve of covectors is simply $\mu_t = \lambda_0$ and $h(t,\tilde u(t),\lambda_0) = 0$.

Recall that the $L$-derivative is obtained by linearising the relation for the Lagrange multiplier rule. But the Lagrange multiplier rule is equivalent to this weak version of the Pontryagin maximum principle, so it is enough to linearise (\ref{eq:hamilt}) and (\ref{eq:maxim}) at $\mu_t = \lambda_0$ and $u(t) = \tilde{u}(t)$. Since $\vec{h}(t,\tilde{u},\lambda_0)= 0$, linearisation of (\ref{eq:hamilt}) gives
$$
\dot{\eta}_t = \left.\frac{\p \vec{h}(t,u,\lambda_0)}{\p u}\right|_{u=\tilde{u}} v_\tau \qquad \iff \qquad \eta_t = \eta_0 + \int_0^t X_\tau v_\tau d\tau, \qquad X_\tau :=   \left.\frac{\p \vec{h}(t,u,\lambda_0)}{\p u}\right|_{u=\tilde{u}}.
$$
Similarly we can linearise the equation (\ref{eq:maxim}) and use the definition of a Hamiltonian vector field to derive:
$$
\left.\frac{\p^2 h(t,u,\lambda_0)}{\p u^2}\right|_{u=\tilde{u}}(v(t),\cdot) + \left\langle d_{\mu_\tau} \left.\frac{\p \vec{h}(t,u,\lambda_0)}{\p u}\right|_{u=\tilde{u}} \cdot , \eta_t \right\rangle = 0  \quad \iff \quad
b_t(v(t),\cdot) + \sigma \left( \eta_t, X_t \cdot \right) = 0,
$$
where $b_t$ is the second derivative of $h(t,u,\lambda_0)$ w.r.t. $u$ at $u = \tilde{u}$. Combining this two expressions we find that a $L$-prederivative is defined as
\begin{align*}
L(\nu J_0^t, E_t)|_V = (P_0^t)_* &\left\{\eta_t = \eta_0 + \int_0^t X_\tau v_\tau d\tau: \eta_0 \in T_{\lambda_0} (T^*_{q_0} M), v_\tau \in V: \right. \\
&\left. \int_0^t \sigma \left( \eta_\tau, X_\tau w_\tau \right) + b_\tau(v_\tau,w_\tau)d\tau = 0 , \forall w_\tau \in V\right\}
\end{align*}

From this description and the definition of the $L$-derivative as a limit of $L$-prederivatives we can deduce an interesting property that can be successfully used to construct approximations for the former. This property says that if we know a $L$-derivative at a moment of time $s_1\in(0,t_1)$ that we denote by $\fL_1$, then the $L$-derivative at a moment of time $s_1 < s_2 \leq t_1$ can be computed using the vectors from $\fL_1$ and variations with support in $[s_1,s_2]$. A precise statement is the following
\begin{lemma}
\label{lemm:add}
Take $0 < s_1  < s_2$ and suppose that $\ind^- Q|_{\ker D_{\tilde{u}}E_t}$ is finite along an extremal curve defined on $[0,s_2]$. Let $\fL_1$ and $\fL_2$ be the two $L$-derivatives for the times $s_1$ and $s_2$ correspondingly. We denote by $V_2$ some finite dimensional subspace of $L^2_k[s_1,s_2]$ and we consider the following equation
\begin{equation}
\label{eq:master_short}
\int_{s_1}^{s_2} \left[ \sigma \left( \lambda + \int_{s_1}^{\tau} X_\theta v_2(\theta) d\theta, X_\tau  w(\tau) \right) + b_\tau\left( v_2(\tau),  w(\tau)\right) \right] d\tau = 0, \qquad \forall w(\tau) \in V_2,
\end{equation}
where $v_2(\tau) \in V_2$, and $\lambda \in \fL_1$. Then $\fL_{2}$ is a generalized limit of the Lagrangian subspaces $L_1^2[V_2]$ defined as
$$
L_1^2[V_2] = \left\{
\lambda + \int_{s_1}^{s_2} X_\tau v(\tau) d\tau : \lambda \in \fL_1 \;,\; v(\tau)\in V_2 \textit{ satisfies (\ref{eq:master_short}) for any } w(\tau)\in V_2 \right\}.
$$
\end{lemma}

This lemma implies that the curve of $L$-derivatives has some sort of a flow property, i.e. the $L$-derivative at the current instant of time can be recovered from the $L$-derivatives at previous moments. This observation is the key moment in our algorithm for computation of the $L$-derivative with arbitrary good precision.

The algorithm can be summarized in the following steps:
\begin{enumerate}
\item Take a partition $0 = s_0 < s_1 < ... < s_N = t$ of the interval $[0,t]$. The finer the partition is, the better will be approximation of the $L$-derivative at the time $t$;
\item Compute inductively $L_{(u,\lambda)}(\nu J_0^{t} ,G_{t})|_{V_i}$, $V_i = \R^{k}\chi_{[s_0,s_1]} \oplus ... \oplus \R^{k}\chi_{[s_i,s_{i+1}]}$ starting from $L_{(u,\lambda)}(\nu J_0^{t} ,G_{t})|_{V_0} = \Pi$.
\end{enumerate}
When $\max |s_{i+1}-s_i| \to 0$, we get in the limit the real $L$-derivative, since piecewise constant functions are dense in $L^2$. This way we reduce the problem to solving iteratively systems of $k$ linear equations. In the Theorem~\ref{thm:algorithm} of the Appendix an explicit solution to this system is given.

This algorithm does not only allow to approximate the $L$-derivative, but  also to compute the index of the Hessian restricted to the subspace of piece-wise constant variations.
\begin{theorem}
\label{thm:maslov_index}
Let $D = \{0 = s_0 < s_1 < ... < s_N  = t_1\}$ be a partition of the interval $[0,t_1]$ and let $V_D$ be the space of piece-wise constant functions with jumps at moments of time $s_i$. We denote by $V_i \subset V_D$ the subspace of functions that are zero for $t > s_i $, $V_i^0 = V_i \cap \ker d_{\tilde{u}} G_{t_1}$ and $\Lambda_i = L(\nu J_0^{t_1} ,G_{t_1})|_{V_i}$
Then the following formula is true
\begin{equation}
\label{eq:maslov_index}
\ind^- Q|_{V_D^0} = \sum_{i=0}^{N} \Ind_\Pi(\Lambda_i, \Lambda_{i+1}) + \dim\left( \bigcap_{i=0}^N \Lambda_i \right) - n,
\end{equation}
where $\Lambda_0 = \Lambda_{N+1} = \Pi$.
\end{theorem}

Moreover one can prove the following result, that is the basis of the whole theory
\begin{theorem}
\label{them:main}
Suppose that $(\tilde{q},\tilde{u})$ is an extremal of the problem (\ref{eq:control}), s.t. the index of the corresponding Hessian is finite and $\fL_t=\fL_{(\tilde u,\lambda_0)}(\nu J_0^t, G_t),\ \fL_t\in L(T_{\lambda_0} (T^* M)),\ t\in[0,t_1]$ is the associate to it family of $L$-derivatives; then $t\mapsto \fL_t$ is a monotone curve and
$$
\ind^- \Hess_{\tilde u} J_0^{t_1}|_{E^{-1} (q_0)} \ge \sup_D \sum_{s_i \in D}\Ind_\Pi(\fL_{s_i},\fL_{s_{i+1}}) + \Ind_\Pi(\fL_{t_1},\fL_{0}) + \dim\left( \bigcap_{t=0}^{t_1} \fL_{t} \right) - n,
$$
where the supremum is taken over all possible finite partitions $D$ of the interval $[0,t_1]$.
\end{theorem}

\subsection{$L$-derivative for problems with a constrained control}
\label{sec:l-deriv_constr}
In our construction of the $L$-derivative we have heavily used the fact that all variations are two-sided, but often in optimal control theory this is not the case. The control parameters may take values in some closed set $U$. Then on the boundary $\p U$ we can only variate along smooth directions of $\p U$. To cover also these kind of situations we are going to use the change of variables introduced in the previous section.

We consider the optimal control problem (\ref{eq:control}), but now we assume that $u \in U \subset \R^k$, where $U$ is a union of a locally finite number of smooth submanifolds $U_i$ without boundaries. In particular, any semi-analytic set is availble.  A typical situation is when the constraints are given by a number of smooth inequalities
$$
p_i(u)\leq 0,
$$
that satisfy
$$
p_i(u) = 0 \qquad \Longrightarrow \qquad d_up_i \neq 0.
$$
For example, $U$ can be a ball or a polytope. In the latter case $U_i$ consists of the interior of polytope and faces of different dimensions.

Recall that in the last subsection we used a time scaling to reduce a free time problem to a fixed one. It is actually very useful to use general time reparameterizations as possible variations, even in the fixed time case. Assume that $t(\tau)$ is an increasing absolutely continuous function, s.t. $t(0) = 0$ and if $t_1$ is fixed also $\tau(t_1) = \tau_1$. Actually instead of the last condition, one can simply take the time variable as a new variable satisfying
$$
\dot{t} = 1.
$$
Then we can consider an optimal control problem
\begin{equation}
\label{eq:control2}
\dot{q} = f(q,u(\tau^{-1}(t))), \qquad \int_0^{t_1} l(q(s),u(\tau^{-1}(s)))ds \to \min.
\end{equation}
which is essentially the optimal control problem (\ref{eq:control}) written in a slightly different way. Since $t(\tau)$ is absolutely continuous, it is of the form
$$
t(\tau) = \int_0^\tau \alpha(s)ds.
$$
We rewrite (\ref{eq:control2}) in the new time $\tau$ to get
$$
\frac{dq}{d\tau} = \alpha(\tau) f(q,u(\tau)), \qquad \int_0^{t_1} \alpha(\tau)l(q(\tau),u(\tau))d\tau \to \min.
$$

Variations with respect to $\alpha(t)$ are called time variations. Since $\alpha>0$, time variations are always two-sided and thus one can include them to study the index of the Hessian via $L$-derivatives. They have been already used to derive necessary and sufficient optimality conditions in the bang-bang case, where no two-sided variations are available if we just vary $u$ (see~\cite{agr_bang,asz_bang}).

Time variations do not give any new contribution to the Morse index of the Hessian if the extremal control $\tilde{u}(t)$ is $C^2$. Indeed, assume for example that $\lambda_t$ is an abnormal extremal corresponding to $\tilde{u}(t)$. Let us denote for simplicity $\beta=1/\alpha$, i.e.
$$
\tau^{-1}(t) = \int_0^t \beta(s) ds
$$
so that we don't have to include differentials of inverse functions in the expressions.

We consider the end-point map $E_t(\tilde u(\beta(s)))$ of (\ref{eq:control2}) and calculate the Hessian with respect to $\beta$ at a point $\beta(s) \equiv 1$. We obtain
$$
\lambda_t d^2_{\beta} E_t (\gamma_1,\gamma_2) = \lambda_t d^2_{\tilde{u}} E_t \left( \frac{d \tilde{u}}{d t}\int_0^t \gamma_1(s) ds, \frac{d \tilde{u}}{d t}\int_0^t \gamma_2(s) ds \right) + \lambda_t d_{\tilde{u}} E_t \left( \frac{d^2 \tilde{u}}{d t}\int_0^t \gamma_1(s) ds \int_0^t \gamma_2(s) ds  \right).
$$
but the the second term is zero since $\tilde{u}$ is extremal and therefore $\lambda_t d_{\tilde{u}}E_t = 0$. This way we see that all the time variations in the Hessian could have been realized by variations of $u$.

If $\tilde{u}(t)$ has less regularity, then the time variations become non-trivial. For example, in the bang-bang case $\tilde{u}$ is piece-wise constant and the effect of the time variations concentrates at the points of discontinuity of $\tilde{u}$. This allows to reduce an infinite dimensional optimization problem to a finite one. This finite dimensional space of variations corresponds simply to variations of the switching times.

If we include the time variations we will have enough two-sided variations to cover all the known cases. It only remains to construct the $L$-derivative over the space of all available two-sided variations. Note that after adding the time variations, this space is not empty. It can be a very difficult computation, but using our algorithm, we can always construct an approximation and obtain a bound on the Morse index.

To apply our algorithm we must define "constant" variations. The set $U$ is a union of smooth submanifolds $U_i\in \R^{k+1}$ without boundaries. Since each $U_i$ is embedded in $\R^{k+1}$ by assumption, we can take the orthogonal projections $\pi^i_u: \R^{k+1} \to T_u U_i$, whenever $u \in U_i$. Then we can define a projection of a general variation $v_t \in L^2_{k+1}[0,t]$ to the subspace of two-sided variations as
$$
\pi_\tau v_\tau = \sum_{i=0}^n \chi_{U_i}(\tau) \pi^i_{\tilde{u}(\tau)} v_\tau,
$$
where $\chi_{U_i}(\tau)$ are the indicator functions. "Constant" variations for the constrained problem are just projections $\pi_\tau v$ of the constant sections $v\in \R^{k+1}$. Equivalently one can consider directly constant variations in $\R^{k+1}$ and simply replace $X_\tau$ in the definition of the $L$-derivative by $X_\tau \pi_\tau$. Then the algorithm from the previous subsection is applicable without any further modifications.

An important remark is that the orthogonal projections depend on the metric that we choose on $\R^{k+1}$. This choice indeed would give us different $L$-prederivatives, since the "constant" variations would be different, but in the limit the $L$-derivative will be the same, because at the end we just approximate the same space in two different ways.

\section*{Appendix: Increment of the index}

Recall that in the Theorem~\ref{thm:maslov_index}, we have stated that by adding piece-wise constant variations, we can track how the Maslov index of the corresponding Jacobi curve changes. But there is no use in this theorem if we are not able to construct explicitly the corresponding $L$-prederivatives from our algorithm. To do this we can use the following theorem.

\begin{theorem}
\label{thm:algorithm}
Suppose that we know $\fL(\nu J_0^t, G_t)|_V$, where $V$ is some space of variations defined on $[0,t]$. We identify $\fL(\nu J_0^t, G_t)|_V$ with $\R^n$ and the space of control parameters with $\R^k$, and put an arbitrary Euclidean metric on both of them. Let $E$ be the space of all $v\in \R^k$ for which
$$
\sigma\left( \eta, \frac{1}{\varepsilon}\int_t^{t+\varepsilon} X_\tau d\tau \cdot v\right) = 0, \qquad \forall \eta \in L(\nu J_0^t, G_t)|_V
$$
and let $L = \fL(\nu J_0^t, G_t)|_V \cap \fL(\nu J_0^t, G_t)|_{\tilde{V}}$, where $ \tilde{V} = V \oplus \R^k \chi_{[t,t+\varepsilon]}$.
We define the two bilinear maps $A: \fL(\nu J_0^t, G_t)|_V  \times E^{\perp} \to \R $, $ Q:  E^{\perp} \times E^{\perp}  \to \R$:
\begin{align*}
A: (\eta, w) &\mapsto \sigma  \left( \eta, \frac{1}{\varepsilon}\int_t^{t+\varepsilon} X_\tau d\tau \cdot w \right), \\
 Q: (v,w) &\mapsto \frac{1}{\varepsilon}\int_{t}^{t+\varepsilon}\sigma \left( \int_t^\tau X_\theta d\theta \cdot v, X_\tau w\right) + b_\tau(v,w) d\tau,
\end{align*}
and we use the same symbols for the corresponding matrices.

Then the new $L$-prederivative $\fL_{(u,\lambda)}(\nu J_0^t, G_t)|_{\tilde{V}}$ is a span of vectors from the subspace  $L$ and vectors
$$
\eta_i + \frac{1}{\varepsilon} \int_t^{t+\varepsilon} X_\tau d\tau \cdot v_i,
$$
where $v_i$ is an arbitrary basis of $E^\perp$ and $\eta_i$ are defined as
$$
\eta_i = - A^+ Q e_i
$$
with $A^+$ being Penrose-Moore pseudoinverse.
\end{theorem}
Although we use some additional structures in the formulation like a Euclidean metric, it will only give a different basis for $\fL(\nu J_0^t, G_t)|_{\tilde{V}}$, but the $L$-derivative itself will be the same.

In general if we would like to compute the difference between indices of the Hessian $\Hess_{\tilde{u}} \varphi|_{\Phi^{-1}(q_0)}$ restricted to two finite-dimensional subspaces $U_1 \subset U$, the Maslov index will only give us a lower bound:
$$
\ind^- Q|_{U^0} - \ind^- Q|_{U_1^0} \geq \Ind_\Pi(L(\varphi,\Phi)|_{U_1},L(\varphi,\Phi)|_{U}),
$$
where as before $U_1^0 = U_1 \cap \ker d_{\tilde{u}} \Phi$ and the same for $U^0$. This formula was proved in~\cite{agr_feedback}.

One can ask, when this inequality becomes an equality. It seems that there is no general if and only if condition, but one can find some nice situations when the equality holds, like in the piece-wise constant case. Another condition that is quite general is stated in the following Theorem.

\begin{theorem}
\label{thm:formula}
Assume that index of the Hessian at a point $(u,\lambda)$ is finite and that we can find a splitting $U_1 \oplus U_2$ of a possibly infinite-dimensional $U$, s.t.
\begin{enumerate}
\item $U_1$ and $U_2$ are orthogonal with respect to $Q$;
\item $Q|_{U_2^0} > 0$, where $U_2^0 = U_2 \cap \ker \Phi'_u$;
\item $\dim \fL_{(u,\lambda)}(\varphi,\Phi)|_{U_1}\cap \Pi = 0$.
\end{enumerate}
Then
$$
\ind^- Q|_{U^0} - \ind^- Q|_{U_1^0} = \Ind_\Pi \left( \fL(\varphi,\Phi)|_{U_1},\fL(\varphi,\Phi)|_{U} \right).
$$

\end{theorem}

We are going to apply twice the Lemma~\ref{lem:important}: first time to the subspace $U_1^0$ in $U^0$ and the second time to $U_2^0$ in $(U_1^0)^\perp$. Assume for now, that $\dim U < \infty$. First we clarify what are all the subspaces presented in the formula. We have
$$
(U_1^0)^\perp = \left\{v = v_1 + v_2 \in U : f'_u(v_1+v_2 )= 0, \langle Qv_1,U_1^0 \rangle = 0 \right\}.
$$
Here we have used our orthogonality assumptions. We claim that $(U_1^0)^\perp$ is actually equal to the subspace
$$
\left\{v_1 + v_2 \in U: f'_u(v_1+v_2 )= 0,\exists \xi_1\in T_{f(u)}^*M, \langle Qv_1 + \xi_1 f'_u,U_1 \rangle = 0\right\}
$$
It is clear that the second space is a subspace of $(U_1^0)^\perp$. We want to prove the converse statement. Assume that $v \in (U_1^0)^\perp$. First we put any Euclidean metric on $T_{f(u)}M$ and use it to define an isomorphism between $T^*_{f(u)}M$ and $T_{f(u)}M$. Secondly we choose a subspace $E \in U_1$ complementary to $U_1^0$ and a basis $e_i$ of $E$, s.t. $f'_u(e_i)$ form an orthogonal subset. Then the covector $\xi_1$ that we need is simply given by
$$
\xi_1 = -\sum_{i=1}^{\dim E} \frac{f'_u (e_i)}{|f'_u(e_i)|^2}\langle Qv_1,e_i \rangle.
$$
Thus the claim has been proved.

From the orthogonality assumption it follows that $U_2^0 \in (U_1^0)^\perp$.
The orthogonal complement of $U_2^0$ in $(U_1^0)^\perp$ is equal to $(U_1^0 + U_2^0)^\perp$, which is equal to
\begin{align*}
(U_1^0 + U_2^0)^\perp = \left\{v_1 + v_2 \in U: f'_u(v_1+v_2 )= 0,\exists \xi_1 \in T_{f(u)}^*M, \langle Qv_1 + \xi_1 f'_u,U_1 \rangle = 0, \langle Qv_2, U_2^0\rangle = 0\right\}.
\end{align*}
Similar to above this is equivalent to
\begin{align*}
(U_1^0 + U_2^0)^\perp = \left\{v_1 + v_2 \in U: f'_u(v_1+v_2 )= 0,\exists \xi_1,\xi_2 \in T_{f(u)}^*M, \right. &\langle Qv_1 + \xi_1 f'_u,U_1 \rangle = 0, \\
& \left. \langle Qv_2 - \xi_2 f'_u, U_2  \rangle = 0\right\}.
\end{align*}

We can now compute the quadratic form $Q$ restricted to $(U_1^0 + U_2^0)^\perp$. Again we use the orthogonality assumption and the equivalent definition of $(U_1^0 + U_2^0)^\perp$ above. Assume that $v = v_1 + v_2 \in (U_1^0 + U_2^0)^\perp$ and $\xi = \xi_1 + \xi_2$. Then
$$
\langle Q (v_1 + v_2),v_1 + v_2 \rangle = \langle Qv_1, v_1\rangle + \langle Qv_2, v_2\rangle = - \xi_1 f'_u v_1 + \xi_2 f'_u v_2 = - \xi f'_u v_1.
$$

Now we would like to write down the expression for the matrix $S$ from the definition of the Maslov index. First we write down the definition of the two $L$-derivatives:
$$
L(\varphi,\Phi)|_{U_1} = \{(\eta_1,f'_u v_1^1): \langle Qv_1^1 + \eta_1 f, U_1\rangle = 0 \};
$$
$$
L(\varphi,\Phi)|_{U} = \{(\eta_2,f'_u (v_1^2 + v_2)): \langle Qv_1^2 + \eta_2 f, U_1\rangle = 0,\langle Qv_2 + \eta_2 f, U_2\rangle = 0 \}.
$$
The quadratic form from the Maslov index is defined on $(L(\varphi,\Phi)|_{U_1}+L(\varphi,\Phi)|_{U})\cap \Pi$. We write $v_1^1 + v_1^2 = v_1$, $\xi_1 = \eta_1 + \eta_2$, $\xi_2 = - \eta_2$ and suppose that $f'_u(v_1) + f'_u(v_2) = 0$. Then for the quadratic form $\tilde{q}$ we have
\begin{align*}
\tilde{q} = \sigma\left( (\eta_1,f'_u v_1^1),(\eta_2,f'_u (v_1^2 + v_2)) \right) = \sigma\left( (\eta_1,f'_u v_1^1),(0,f'_u v_2) \right) = \eta_1 f'_u v_2 = -\eta_1 f'_u v_1 = -\xi f'_u v_1
\end{align*}
In the second equality we have used that $(\eta_1,f'_u v_1^1)$ and $(\eta_2,f'_u v_1^2)$ belong to $L(\varphi,\Phi)|_{U_1}$ by definition.

We see that this gives the same expression as for $Q|_{(U_2^0)^\perp}$. But moreover both quadratic forms are actually defined on the same space. Indeed, we have

\begin{align*}
(L(\varphi,\Phi)|_{U_1}+L(\varphi,\Phi)|_{U})\cap \Pi = \{(&\xi_1 + \xi_2,0):\exists v_i\in V_i, f'_u(v_1 + v_2) = 0, (\xi_1,f'_u v_1) \in L(\varphi,\Phi)|_{U_1}, \\
(-&\xi_2,f'_u v_2) \in L(\varphi,\Phi)|_{U_2} \} = (L(\varphi,\Phi)|_{U_1}+L(\varphi,\Phi)|_{U_2})\cap \Pi
\end{align*}
But if we add to $(\xi_1 + \xi_2) \in (L(\varphi,\Phi)|_{U_1}+L(\varphi,\Phi)|_{U})\cap \Pi$ the corresponding $v_i$ and to $v_i \in U_i \cap (U_1^0 + U_2^0)^\perp$ the corresponding $\xi_1 + \xi_2 $, we obtain the same space.

Now we compute the other terms from the formula in Lemma~\ref{lem:important}. We have
$$
U_1^0 \cap (U_1^0)^\perp = \{v_1  \in U_1^0: Q(v_1,U_1^0)=0\}.
$$
Similarly to the discussion in the beginning of the proof, we can show that
$$
U_1^0 \cap (U_1^0)^\perp = \{v_1  \in U_1^0: \exists \xi \in T_{f(u)}^*M, Q(v_1 +\xi f'_u,U_1)=0\}.
$$
We do now the same for $\ker Q|_{U^0} \cap U_1^0$:
$$
\ker Q|_{U^0} \cap U_1^0 = \{v_1 \in U_1^0: \langle Q v_1, U^0 \rangle = 0\} = \{v_1 \in U_1^0: \langle Q v_1 + \xi f'_u, U \rangle = 0\}
$$

To understand the dimensions, we look carefully at the equation
$$
\langle Qv_1 +\xi f'_u,U_1\rangle=0
$$
If there are two solutions $(\xi,v_1)$ and $(\xi,v'_1)$ of this equation, then by linearity $(0,v_1 - v'_1)$ is a solution as well and thus all solutions are uniquely defined by different $\xi$ modulo $\ker Q|_{U_1} \cap U_1^0$. These $\xi$ lie in $L(\varphi,\Phi)|_{U_1} \cap \Pi$ as can be seen from the definitions. Therefore
$$
\dim \left( U_1^0 \cap (U_1^0)^\perp \right) = \dim \left( L(\varphi,\Phi)|_{U_1} \cap \Pi\right) + \dim \left( \ker Q|_{U_1} \cap U_1^0\right)
$$

Now we do the same for
$$
0 = Q(v_1 +\xi f'_u,V)= Q(v_1 +\xi f'_u,U_1) + \xi f'_u U_2
$$
Again $\xi$ are defined uniquely modulo $\ker Q|_{U_1} \cap U_1^0$, but now they lie in $L(\varphi,\Phi)|_{U} \cap \Pi$. Therefore
$$
\dim \left( \ker Q|_{U^0} \cap U_1^0 \right) = \dim \left( L(\varphi,\Phi)|_{U} \cap \Pi\right) + \dim \left( \ker Q|_{U_1} \cap U_1^0\right).
$$

Since $Q$ is positive on $U_2^0$, we have $(U_2^0)^\perp \cap U_2^0  = \{0\}$ and so we can collect all the formulas using the fact that $ (L(\varphi,\Phi)|_{U} \cap \Pi) \subset (L(\varphi,\Phi)|_{U_1} \cap \Pi) $:
\begin{align*}
\ind^- Q|_{U^0} - \ind^- Q|_{U_1^0} &= \Ind_{\Pi}\left( L(\varphi,\Phi)|_{U_1},L(\varphi,\Phi)|_{U} \right) +\frac{1}{2}\dim \left( L(\varphi,\Phi)|_{U_1} \cap \Pi \right) - \\
&-\frac{1}{2}\left( \dim L(\varphi,\Phi)|_{U}\cap \Pi \right)
\end{align*}

Under the assumption three the formula is valid also in the infinite dimensional case. We know that the $L$-prederivatives will converge and that the quadratic form from the Maslov-type index is continuous. The only possibly discontinuous term are the dimensions of various intersections, but they are zero now for $L$-prederivatives close to the $L$-derivatives.

\bibliographystyle{plain}
\bibliography{bib}

\end{document}